\documentclass[10pt,oneside,a4paper]{amsart}
\usepackage{latexsym}
\usepackage{amsfonts,amsmath,amssymb,indentfirst}
\usepackage[english]{babel}
\usepackage[pdftex]{hyperref}
\newcommand{\bdism}{\begin{displaymath}}
\newcommand{\edism}{\end{displaymath}}
\newcommand{\cc}{\mathbb{C}}
\newcommand{\rr}{\mathbb{R}}
\newcommand{\qq}{\mathbb{Q}}

\newcommand{\pp}{\mathbb{P}}
\newcommand{\oo}{\mathcal{O}}

\newcommand{\dd}{\mathfrak{D}}
\newcommand{\bb}{\mathfrak{B}}
\DeclareMathOperator{\vol}{vol} 
\DeclareMathOperator{\supp}{Supp} 
\DeclareMathOperator{\bs}{Bs} \DeclareMathOperator{\B}{\mathbf{B}}
\newtheorem{theorem}{Theorem}[section]
\newtheorem{proposition}[theorem]{Proposition}
\newtheorem{corollary}[theorem]{Corollary}
\newtheorem{lemma}[theorem]{Lemma}
\newtheorem{remark}[theorem]{Remark}
\newtheorem{example}[theorem]{Example}
\newtheorem{definition}[theorem]{Definition}
\newtheorem{question}[theorem]{Question}
\newtheorem{conjecture}[theorem]{Conjecture}
\newtheorem{fact}[theorem]{Fact}

\address{Department of Mathematics, Princeton University, NJ 08544-1000,
USA} \email{gdi@math.princeton.edu}

\address{Department of Mathematics, Duke University, Durham NC 27708-0320,
USA} \email{luca@math.duke.edu}

\author{Gabriele Di Cerbo and Luca F. Di Cerbo}

\title{\bf Positivity in K\"ahler-Einstein theory}

\begin{document}
\pagestyle{headings}

\begin{abstract}
Tian initiated the study of incomplete K\"ahler-Einstein metrics on
quasi-projective varieties with cone-edge type singularities along a
divisor, described by the cone-angle $2\pi(1-\alpha)$ for $\alpha\in
(0, 1)$. In this paper we study how the existence of such
K\"ahler-Einstein metrics depends on $\alpha$. We show that in the
negative scalar curvature case, if such K\"ahler-Einstein metrics
exist for all small cone-angles then they exist for every
$\alpha\in(\frac{n+1}{n+2}, 1)$, where $n$ is the dimension. We also
give a characterization of the pairs that admit negatively curved
cone-edge K\"ahler-Einstein metrics with cone angle close to $2\pi$.
Again if these metrics exist for all cone-angles close to $2\pi$,
then they exist in a uniform interval of angles depending on the
dimension only. Finally, we show how in the positive scalar
curvature case the existence of such uniform bounds is obstructed.
\end{abstract}

\maketitle

\tableofcontents

\section{Introduction}\label{introduction}
\pagenumbering{arabic}

 G.~Tian \cite{Tian} outlined a very general program
for the construction of \emph{incomplete} K\"ahler-Einstein metrics
on quasi-projective varieties $X\setminus D$ where $X$ is a smooth
projective variety and $D\subset X$ is a simple normal crossing
divisor. Such metrics have cone-edge like singularities along $D$,
their asymptotic behavior is described by the cone-angle
$2\pi(1-\alpha)$ for $\alpha\in(0, 1)$. These metrics have been
extensively studied, see for example \cite{CDS}, \cite{Tian}, \cite{Donaldson},
\cite{Berman}, \cite{Brendle}, \cite{Campana1}, \cite{Mazzeo},
\cite{Mazzeo2}. Recently, Jeffres-Mazzeo-Rubinstein \cite{Mazzeo}
and Mazzeo-Rubinstein \cite{Mazzeo2} made substantial progress
toward the completion of Tian's original program. In particular,
they show that such metrics, with negative scalar curvature and cone
angle $2\pi(1-\alpha)$, exist when $K_X+\alpha D$ is ample, see also
\cite{Campana1} for the case $\alpha\in[\frac{1}{2}, 1)$.

In this paper we study geometric problems motivated by these
results.

\begin{itemize}
\item Characterize those pairs $(X,D)$ for which
a K\"ahler-Einstein metric with cone angle $2\pi(1-\alpha)$ exists
for some $\alpha$.
\item Given $(X,D)$, describe the values $\alpha$ for which
such a metric exists.
\end{itemize}

We have a satisfactory answer to these questions for small
cone-angles in the negative scalar curvature case.

\begin{theorem}\label{th1}
Let $X$ be a smooth,  $n$-dimensional projective variety and
$D\subset X$ a simple normal crossing divisor. The following are
equivalent.
\begin{enumerate}
\item  $(X, D)$ admits negative
K\"ahler-Einstein metrics with cone-edge singularities along $D$
with cone angles $2\pi(1-\alpha)$ for all $\alpha\in (1-\epsilon,
1)$ and some $\epsilon>0$.
\item  $(X, D)$ admits negative
K\"ahler-Einstein metrics with cone-edge singularities along $D$
with cone angles $2\pi(1-\alpha)$ for every
 $\alpha\in (\frac{n+1}{n+2}, 1)$.
\item The self-intersection $\bigl(K_{X}+D\bigr)^n$ is positive,
$(K_{X}+D)\cdot C\geq 0$ for every curve $C\subset X$ and  there are
no curves $C\subset X$ such that $(K_{X}+D)\cdot C=0$ and
$K_{X}\cdot C\leq 0$.
\end{enumerate}
\end{theorem}

Similarly, we can threat the negative scalar curvature case for
cone-angles close to $2\pi$.

\begin{theorem}\label{th2}
Let $X$ be a smooth, $n$-dimensional projective variety of general
type and $D\subset X$ a simple normal crossing divisor. The
following are equivalent.
\begin{enumerate}
\item  $(X, D)$ admits negative
K\"ahler-Einstein metrics with cone-edge singularities along $D$
with cone angles $2\pi(1-\alpha)$ for all $\alpha\in (0, \epsilon)$
and some $\epsilon>0$.
\item  $(X, D)$ admits negative
K\"ahler-Einstein metrics with cone-edge singularities along $D$
with cone angles $2\pi(1-\alpha)$ for every $\alpha\in (0,
\frac{1}{2n+1})$.
\item $K_{X}\cdot C\geq 0$ for every curve $C\subset X$ and there are no
curves $C\subset X$ such that $(K_{X}+D)\cdot C\leq0$ and
$K_{X}\cdot C=0$.
\end{enumerate}
\end{theorem}

The proofs of these results rely on the cone theorem and on the base
point free theorem. For the technical details see Section
\ref{General}.

Remarkably, in the positive scalar curvature case the
existence of such uniform bounds is obstructed for both small and
large cone-angles, see Examples \ref{ex1} and \ref{ex2}.
Nevertheless, we are able to quantify this obstruction when
$-K_{X}-D$ and $-K_{X}$ are ample. For small cone-angles we obtain
that this obstruction depends \emph{only} on the self-intersection
$(-K_{X}-D)^{n}$, compare with Example \ref{ex1}. More precisely, we
can state the following.

\begin{theorem}\label{logf}
Let $\bb$ be a set of pairs $(X,D)$ such that $X$ is a smooth
projective variety of dimension $n$ and $D$ is a reduced effective
divisor with simple normal crossing support such that $-(K_{X}+D)$
is ample. The following are equivalent.
\begin{enumerate}
\item There exists a positive number $\alpha_{0}<1$ such that
$-(K_{X}+\alpha D)$ is ample for any $\alpha\in(\alpha_{0},1]$ and
any $(X,D)\in\bb$.
\item There exists a positive integer $M_{0}$ such that $(-(K_{X}+D))^{n}\leq M_{0}$ for any $(X,D)\in\bb$.
\end{enumerate}
\end{theorem}

We can also threat the case of cone-angles close to $2\pi$. In this
case the obstruction depends \emph{only} on the intersection number
$D\cdot (-K_{X})^{n-1}$, compare with Example \ref{ex2}.

\begin{theorem}\label{logf2}
Let $\dd$ be a set of pairs $(X,D)$ such that $X$ is a smooth
projective Fano variety of dimension $n$ and $D$ is a reduced
effective divisor with simple normal crossing support. The following
are equivalent.
\begin{enumerate}
\item There exists a positive number $\alpha_{0}<1$ such that
$-(K_{X}+\alpha D)$ is ample for any $\alpha\in[0,\alpha_{0})$ and
any $(X,D)\in\dd$.
\item There exists a positive integer $M_{0}$ such that $D\cdot(-K_{X})^{n-1}\leq M_{0}$ for any $(X,D)\in\dd$.
\end{enumerate}
\end{theorem}

The strategies used in the proofs of Theorems \ref{logf} and
\ref{logf2} are completely different from the techniques employed in
the negative scalar curvature case. The proofs of these theorems
rely on effective very ampleness results as well as on boundness of
Fano varieties. For the technical details see Section \ref{Fano}.\\

We now recall some of the key features of Tian's program \cite{Tian}
and their relation with the theme of positivity in algebraic
geometry. Incomplete K\"ahler-Einstein metrics are geometrically
very interesting. Unlike the complete case, they have a much more
interesting moduli theory as they are not subject to Yau's
generalized Schwartz lemma. Furthermore, in Tian's original program
\cite{Tian}, they should also be used to better understand the
complete K\"ahler-Einstein metrics previously constructed on $X$ and
$X \backslash D$. Recall that T. Aubin and S.-T. Yau constructed
negatively curved K\"ahler-Einstein metrics on $X$ when $K_{X}$ is
ample, see \cite{Aubin} and \cite{Yau}. Similarly, there are many
existence results for complete negatively curved K\"ahler-Einstein
metrics on $X \backslash D$ under certain positivity assumption for
$K_{X}+D$, see for example \cite{TianY}, \cite{Yau2}, \cite{Wu1},
\cite{Wu2}. For precise definitions and details see Section
\ref{Kahler-Einstein}. Now, one should be able to recover these
metrics by taking limits of K\"ahler-Einstein metrics with cone-edge
singularities as the cone angle goes to $2\pi$ or $0$. Moreover, the
study of this limiting behavior should give a better understanding
of the asymptotic behavior of the complete K\"ahler-Einstein metrics
on $X \backslash D$.

K\"ahler-Einstein metrics with cone-edge singularities are also very
interesting from a \emph{positivity} point of view. Their existence
implies strong positivity properties for $\rr$-divisors of the form
$K_{X}+\alpha D$ or $-(K_{X}+\alpha D)$ for some $\alpha\in(0, 1)$.
We refer to Section \ref{Kahler-Einstein} for the precise
definitions, results and to Proposition \ref{luca} for the necessary
condition for the existence of such metrics. Thus, the existence of
these metrics appears to be associated with the positivity of the
log-canonical or anti-log-canonical divisor of a Kawamata
log-terminal (klt) pair rather than a log-canonical (lc) pair. This
fact turns out to be extremely useful since the \emph{base-point
free} theorem is currently known for big, nef and klt pairs only. We
exploit systematically this point in Section \ref{General}. For the
definitions of klt and lc pairs, the statement of the base-point
free theorem and the results from algebraic geometry used in the
rest of this work we refer to Section \ref{Geometry}.

To sum up, it is important to study which positivity properties are
preserved, lost or gained by passing from a lc pair $(X, D)$ to a
klt pair $(X, \alpha D)$, $\alpha\in(0, 1)$, and viceversa. We find
that these questions can be approached by studying the behavior of
certain positivity thresholds, see Sections \ref{General},
\ref{Modulo} and \ref{Fano}. For example, given a pair $(X, D)$ with
$K_{X}+D$ big and nef the key object to study is the so called
\emph{nef threshold} $r(X, D)$, see Definition \ref{threshold} and
the results in Section \ref{General}. All the questions we address
here are explicitly formulated throughout the text, see Questions
\ref{q1}, \ref{q14}, \ref{q3} and \ref{qfano}. We provide complete
solutions to all of them through Theorems \ref{tian}, \ref{tian2},
Example \ref{ex1} and Theorems \ref{logf}, \ref{logf2}. These
results seem to have interest of their own. Moreover, they have
applications beyond the study of K\"ahler-Einstein metrics, see for example \cite{DD15}.

The paper is organized as follows. In Section \ref{Kahler-Einstein},
we recall the theory of K\"ahler-Einstein metrics on
quasi-projective varieties. We discuss the existence of both
complete and incomplete K\"ahler-Einstein metrics of negative and
positive scalar curvature. It is shown how the existence of such
metrics motivates a number of natural positivity questions. These
questions are thoroughly studied in Sections \ref{General},
\ref{Modulo} and \ref{Fano}. In Section \ref{Geometry}, we state
some fundamental results from the theory of the minimal model which
we repeatedly use in this work. In Section \ref{Modulo}, we study
the notion of ampleness for line bundles on quasi-projective
varieties. This analysis is important for the study of negatively
curved K\"ahler-Einstein metrics beyond the cone-edge asymptotic,
see Corollary \ref{semipositive} and the remarks which follows.
Finally, all the results are motivated and supported with examples. \\

\noindent\textbf{Acknowledgements}. The first named author would
like to express his gratitude to Professor J\'anos Koll\'ar for his
constant support and for very useful comments on a preliminary
version of this work. He also would like to thank Yoshinori Gongyo
and Chenyang Xu for some useful suggestions. The second named author
would like to thank Professor Mark Stern for bringing \cite{Mumford}
to his attention and for several useful discussions and comments.

\section{K\"ahler-Einstein metrics on quasi-projective
varieties}\label{Kahler-Einstein}

In this section we recall and collect some generalities regarding
the existence of both complete and incomplete K\"ahler-Einstein
metrics on quasi-projective varieties. In particular, we discuss the
\emph{necessary} conditions for the existence of such metrics. This
analysis uses the language and the basic theory of \emph{K\"ahler
currents} on K\"ahler manifolds \cite{Demailly2}. The results
obtained motivate the positivity questions addressed in the rest of
this work.

Let us start by fixing notation and by giving some definitions. Let
$X$ be a $n$-dimensional projective manifold and $D=\sum_{i}D_{i}$
be a reduced simple normal crossing divisor. Concretely, this simply
means that the irreducible complex hypersurfaces $D_{i}$ are smooth
and that they intersect transversally. Given the pair $(X, D)$, let
us recall the notion of a K\"ahler metric with Poincar\'e type
singularities along $D$. A smooth K\"ahler metric $\hat{\omega}$ on
$X\backslash D$ is said to have Poincar\'e singularities along the
$D_{i}$'s if for any point $p\in D$ and coordinate neighborhood
$(\Omega;z_{1},...,z_{n})$ centered at $p$ for which

\begin{align}\notag
D_{|_{\Omega}}=\left\{z_{1}\cdot...\cdot z_{k}=0\right\}
\end{align}

then $\hat{\omega}$ is \emph{quasi-isometric} in $\Omega$ to the
following model metric
\begin{align}\label{model1}
\omega_{0}=\sqrt{-1}\bigg(\sum^{k}_{i=1}\frac{dz_{i}\wedge
d\overline{z}_{i}}{|z_{i}|^{2}\log^{2}|z_{i}|^{2}}+\sum^{n}_{i=k+1}dz_{i}\wedge
d\overline{z}_{i}\bigg).
\end{align}

It is easy to construct K\"ahler metrics with Poincar\'e
singularities along a divisor.

\begin{example}
Let $\omega$ be a K\"ahler metric on $X$. For simplicity let $D$ be
a smooth reduced divisor. Choose an Hermitian metric $\|\cdot\|$ on
$\mathcal{O}_{X}(D)$. Let $s\in H^{0}(X, \mathcal{O}(D))$ be a
defining section for the divisor $D$. Then for $T>0$ big enough
\begin{align}\notag
\hat{\omega}=T\omega-\sqrt{-1}
\partial\overline{\partial}\log\log^{2}\|s\|^{2}
\end{align}
defines K\"ahler metric on $X\backslash D$ with Poincar\'e
singularities along $D$.
\end{example}

Note that any metric on $X\backslash D$ which is quasi-isometric to
the model given in \ref{model1} is necessarily a complete Riemannian
metric. In fact, it suffices to show that the boundary divisor $D$
is ``metrically at infinity'', which is clearly the case here since
\begin{align}\notag
\int^{l}_{0}\frac{dr}{r|\log{r}|}=-\log{\log{\frac{1}{r}}}\Big{|}^{l}_{0}=\infty.
\end{align}

Moreover, such metrics have finite volume as one can easily see
computing
\begin{align}\notag
\int^{2\pi}_{0}\int^{l}_{0}\frac{rdrd\theta}{r^{2}(\log{r})^{2}}<\infty.
\end{align}

If we now restrict our attention to K\"ahler metrics on $X\backslash
D$ which have Poincar\'e singularities along $D$ then they can also
be regarded as \emph{global} objects on $X$ as \emph{currents}. For
the basic definitions and properties of complex currents we refer to
the books \cite{Griffiths} and \cite{Demailly2}.

\begin{proposition}\label{exstension}
Let $\hat{\omega}$ be a K\"ahler metric on $X\backslash D$ with
Poincar\'e singularities along D. Then $\hat{\omega}$ defines a
K\"ahler current on $X$.
\end{proposition}

\begin{proof}
Recall that in order to extend a positive closed $(1, 1)$-current
across an analytic set it suffices to check that it has finite mass
near it. For a proof of this important fact see Th\'eor\`eme 1.1 in
\cite{Sibony} and the bibliography there. Now, this amounts to the
fact that the Poincar\'e metric has finite volume. Finally,
$\hat{\omega}$ defines a K\"ahler current since the condition given
in \ref{model1} ensures that it dominates a positive $(1, 1)$-form
on $X$.
\end{proof}

\begin{remark}
It seems that the idea of bounding forms on quasi-projective variety
by Poincar\'e type metrics goes at least back to Cornalba and
Griffiths \cite{Cornalba}. Moreover, many of the extension
properties of such forms were also explained by Mumford
\cite{Mumford}.
\end{remark}

We now return to the existence problem for K\"ahler-Einstein metric
with Poincar\'e type singularities. Note that, classical theorems in
global Riemannian geometry \cite{Petersen}, can be used to rule out
the existence of such Einstein metrics when the scalar curvature is
non-negative.

It is easy to construct K\"ahler-Einstein metrics with Poincar\'e
singularities along a divisor and negative scalar curvature.

\begin{example}\label{exampleR}
Let $\Sigma$ be a finite volume hyperbolic Riemannian surface and
let $\overline{\Sigma}=\Sigma\cup\{p_{1}, ..., p_{k}\}$ be its
natural compactification. Then the hyperbolic metric on $\Sigma$ is
a K\"ahler-Einstein metric with Poincar\'e singularities along the
$p_{i}$'s.
\end{example}

Note that, in Example \ref{exampleR}, the \emph{logarithmic}
canonical bundle
\begin{align}\notag
K_{\overline{\Sigma}}+P, \quad P:=\{p_{1}, ..., p_{k}\},
\end{align}
associated to the pair $(\overline{\Sigma}, P)$ is an \emph{ample}
divisor on $\overline{\Sigma}$. Recall that $\Sigma$ admits a
hyperbolic metric if and only if $2g(\overline{\Sigma})-2+k>0$, see
for example \cite{Forster}. For the basic results in logarithmic
geometry we refer to \cite{Iitaka}.

Now, the positivity of the logarithmic canonical line bundle is not
unexpected. In fact, given a logarithmic pair $(X, D)$ with $D$ a
simple normal crossing divisor, the existence of a K\"ahler-Einstein
metric with negative scalar curvature on $X\backslash D$ and
Poincar\'e singularities along $D$ is guaranteed under the condition
that $K_{X}+D$ is an ample divisor; see the works of R. Kobayashi
\cite{kobayashi} and Cheng-Yau \cite{Cheng}. Let us briefly discuss
the proof of such results. This existence problem reduces to the
study of a complex Monge-Amp\`ere equation of the form
\begin{align}\label{degenerate}
(\omega+\sqrt{-1}\partial\overline{\partial}\varphi)^{n}=
e^{\varphi}\Psi
\end{align}
where
\begin{align}\notag
\Psi=\frac{\Omega}{\prod_i\|s_{i}\|^{2}(\log\|s_{i}\|^{2})^{2}}
\end{align}
for some globally defined volume form $\Omega$ on $X$ and where
$\omega$ is a Carlson-Griffiths \cite{Carlson} K\"ahler metric on
$X\backslash D$. More precisely, $\omega$ is a complete K\"ahler
metric on $X\backslash D$ with Poincar\'e singularities along $D$
which is obtained as minus the Ricci curvature of $\Psi$. As
explained in \cite{Carlson}, if we assume the log-canonical divisor
to be ample, this is always possible by appropriately choosing
$\Omega$ and the Hermitian metrics $\|\cdot\|_{i}$ on
$\mathcal{O}_{X}(D_{i})$. Now, as proved in \cite{kobayashi} and
\cite{Cheng} the solution $\omega_{\varphi}$ of \ref{degenerate} has
Poincar\'e type singularities. Thus, the Poincar\'e-Lelong formula
\cite{Griffiths} gives the following.

\begin{remark}\label{motivation}
The K\"ahler current $\omega_{\varphi}$ solving \ref{degenerate}
satisfies a distributional equation of the form
\begin{align}\notag
Ric_{\omega_{\varphi}}-\sum_{i}[D_{i}]=-\omega_{\varphi},
\end{align}
where by $[D]$ we indicate the current of integration along $D$.
\end{remark}

It is easy to construct Carlson-Griffiths type metrics under
slightly different positivity conditions. For example, let us
consider a collection of positive numbers $\alpha_{i}\in(0, 1]$ and
assume the \emph{twisted} log-canonical divisor
\begin{align}\notag
K_{X}+\sum_{i}\alpha_{i}D_{i}
\end{align}
to be ample on $X$. Then, a simple modification of the original
argument given in \cite{Carlson} can be used to construct a K\"ahler
metric on $X\backslash D$ with Poincar\'e singularities along $D$
which is in the cohomology class of the $\rr$-divisor
$K_{X}+\sum_{i}\alpha_{i}D_{i}$. As in the untwisted case, one can
then try to deform this background metric to a complete
K\"ahler-Einsten metric. Nevertheless, this approach seems to not
work in this context. In fact, K\"ahler-Einstein metrics with
Poincar\'e type singularities are complete and therefore quite rigid
because of Yau's generalized Schwartz lemma. This is why, in the
definition below, we limit ourselves to the case where all the
$\alpha_{i}$'s are identically equal to one. Compare with Definition
\ref{TianDef} below.

\begin{definition}\label{Tian-Yau-Wu}
A K\"ahler current $\hat{\omega}$ with Poincar\'e singularities
along the $D_{i}$'s is called K\"ahler-Einstein if it satisfies a
distributional equation of the form
\begin{align}\label{Tian-Wu}
Ric_{\hat{\omega}}-\sum_{i}[D_{i}]=-\hat{\omega}
\end{align}
where by $[D]$ we indicate the current of integration along $D$.
\end{definition}

We can now derive the necessary condition for the existence of a
K\"ahler-Einstein metric with Poincar\'e singularities.

\begin{proposition}\label{poincare}
Let $\hat{\omega}$ be a K\"ahler-Einstein metric with Poincar\'e
singularities as in \ref{Tian-Wu}, then the divisor
$K_{X}+\sum_{i}D_{i}$ is ample.
\end{proposition}
\begin{proof}
Since by assumption $\hat{\omega}$ satisfies \ref{Tian-Wu}, we have
that the cohomology class of the divisor $K_{X}+\sum_{i}D_{i}$ can
be represented by a K\"ahler current. By the structure of the
\emph{pseudo-effective} cone given in \cite{Demailly1}, we conclude
that $K_{X}+\sum_{i}D_{i}$ is a \emph{big} divisor. Next, we want to
show that $K_{X}+\sum_{i}D_{i}$ has to be an ample divisor. First, a
simple computation shows that the \emph{Lelong} numbers of
$\hat{\omega}$ are zero at any smooth point of $D$. Then a
regularization argument based on results of Demailly
\cite{Demailly1} can be used to show that $K_{X}+\sum_{i}D_{i}$ is
indeed an ample divisor.
\end{proof}

We now briefly discuss the existence problem for incomplete
K\"ahler-Einstein metrics on quasi-projective varieties. The most
natural kind of incomplete metrics appearing in K\"ahler geometry
seem to have \emph{cone-edge} singularities. The study of such
metrics was originally proposed by G. Tian in \cite{Tian}. Thus,
given a pair $(X, D)$ as above, let us recall the notion of a
K\"ahler metric with edge singularities along $D$. Let $\{D_{i}\}$
be the irreducible smooth components of $D$ and consider a
collection of positive numbers $\alpha_{i}\in(0, 1)$. A smooth
K\"ahler metric $\hat{\omega}$ on $X\backslash D$ is said to have
edge singularities of cone angles $2\pi(1-\alpha_{i})$ along the
$D_{i}$'s if for any point $p\in D$ and coordinate neighborhood
$(\Omega;z_{1},...,z_{n})$ centered at $p$ for which
\begin{align}\notag
D_{|_{\Omega}}=\left\{z_{1}\cdot...\cdot z_{k}=0\right\},
\end{align}
then $\hat{\omega}$ is \emph{quasi-isometric} in $\Omega$ to the
following model metric
\begin{align}\label{model}
\omega_{0}=\sqrt{-1}\bigg(\sum^{k}_{i=1}\frac{dz_{i}\wedge
d\overline{z}_{i}}{|z_{i}|^{2\alpha_{i}}}+\sum^{n}_{i=k+1}dz_{i}\wedge
d\overline{z}_{i}\bigg).
\end{align}

It is easy to construct K\"ahler metrics with edge singularity along
a divisor.

\begin{example}
Let $\omega$ be a K\"ahler metric on $X$. For simplicity let $D$ be
a smooth reduced divisor and choose a real number $\alpha\in(0, 1)$.
Choose an Hermitian metric $\|\cdot\|$ on $\mathcal{O}_{X}(D)$. Let
$s\in H^{0}(X, \mathcal{O}(D))$ be a defining section for the
divisor $D$. Then for $T>0$ big enough
\begin{align}\notag
\hat{\omega}=T\omega+\sqrt{-1}
\partial\overline{\partial}\|s\|^{2(1-\alpha)}
\end{align}
defines K\"ahler metric on $X\backslash D$ with edge singularities
of cone angle $2\pi(1-\alpha)$ along $D$.
\end{example}

As first observed in \cite{Jeffress}, $\hat{\omega}$ is indeed an
incomplete K\"ahler metric of finite volume on $X\backslash D$. Now,
because of the finite volume property, standard results in the
theory of currents \cite{Demailly2} imply that $\hat{\omega}$ can be
regarded as a \emph{K\"ahler current} on $X$ with \emph{singular
support} $D$. The fact that $\hat{\omega}$ defines a K\"ahler
current and not only a closed positive current follows from the
quasi-isometric condition given in \ref{model}.

We can now introduce the definition of a K\"ahler-Einstein metric
with edge singularities.

\begin{definition}[Tian \cite{Tian}]\label{TianDef}
A K\"ahler current $\hat{\omega}$ with edges of cone angles
$2\pi(1-\alpha_{i})$ along the $D_{i}$'s is called K\"ahler-Einstein
with curvature $\lambda$ if it satisfies the distributional equation
\begin{align}\label{Tian}
Ric_{\hat{\omega}}-\sum_{i}\alpha_{i}[D_{i}]=\lambda\hat{\omega}
\end{align}
where by $[D]$ we indicate the current of integration along $D$.
\end{definition}

A simple regularization argument based on results of Demailly
\cite{Demailly1} can now be used to derive the \emph{necessary}
condition for the existence of a K\"ahler-Einstein metric with edge
singularities.

\begin{proposition}\label{luca}
Let $\hat{\omega}$ be a K\"ahler-Einstein metric with edge
singularities as in \ref{Tian} with $\lambda=-1$  $(\lambda=1)$,
then the $\rr$-divisor $K_{X}+\sum_{i}\alpha_{i}D_{i}$
$(-K_{X}-\sum_{i}\alpha_{i}D_{i})$ is ample.
\end{proposition}
\begin{proof}
See Propositions 2.1 and 2.2 in \cite{LucaE}.
\end{proof}

To sum up, the entire set of questions concerning the existence and
limiting behavior of families of K\"ahler-Einstein metrics on
quasi-projective varieties requires a precise understanding of the
positivity properties of certain log-canonical bundles. This is the
object of study in Sections \ref{General}-\ref{Fano}. Since we are
primarily interested in taking limits as the cone angles go to $0$
or $2\pi$ we limit ourselves to the case where all the
$\alpha_{i}$'s are identical.

\section{Some theorems from algebraic geometry}\label{Geometry}

In this section we recall some fundamental results in algebraic
geometry. We follow the notation and terminology of \cite{KM}.
However we state here some definitions we will need later.
\begin{definition}
A 1-cycle is a formal linear combination of irreducible, reduced and
proper curves $C=\sum a_{i}C_{i}$. Two 1-cycles $C,C'$ are called
numerically equivalent if $(C\cdot D)=(C'\cdot D)$ for any Cartier
divisor $D$. 1-cycles with real coefficients modulo numerical
equivalence form a $\rr$-vector space and we denote it by
$N_{1}(X)$.
\end{definition}
We are interested in particular subspaces of $N_{1}(X)$.
\begin{definition}
Let $X$ be a proper variety. \bdism NE(X):=\left\{ \sum a_{i}[C_{i}]
\:|\: C_{i}\subset X, \: 0\leq a_{i}\in\rr\right\}\subset N_{1}(X).
\edism
Let $\overline{NE}(X)$ be the closure of $NE(X)$ in $N_{1}(X)$. \\
For any divisor $D$, set $D_{\geq0}:=\left\{x\in N_{1}(X)\:|\:
D\cdot x\geq0\right\}$ and similarly for $>0$, $\leq0$ and $<0$.
Finally \bdism \overline{NE}(X)_{D\geq0}:=\overline{NE}(X)\cap
D_{\geq0}. \edism
\end{definition}

In birational geometry and especially in the minimal model program,
we need to deal with some classes of singularities of pairs $(X,D)$.
The definitions are a little bit technical but the idea is simple.
For a singular variety $X$ we want to measure its singularities
comparing $K_{X}$ and $f^{*}K_{X}$, where $f$ is a resolution of the
singularities of $X$. In case we have a pair $(X,D)$ we will measure
the singularities of $X$ and $D$ together.

\begin{definition}
A pair $(X,D)$ consists of a normal variety $X$ and a $\rr$-Weil
divisor $D$ such that $K_{X}+D$ is $\rr$-Cartier.
\end{definition}

Let $(X,D)$ be a pair. A log resolution of the pair is a birational
map $f:Y\rightarrow X$ such that $Y$ is smooth, the exceptional set
$Ex(f)$ of $f$ is a divisor and $Ex(f)\cup f^{-1}(\supp D)$ is a
simple normal crossing divisor. For any log resolution of $(X,D)$ we
can write $K_{Y}+\Gamma\sim_{\qq}f^{*}(K_{X}+D)$, where $\Gamma=\sum
a_{i}\Gamma_{i}$ and $\Gamma_{i}$ are distinct reduced irreducible
divisors. $1-a_{i}$ is called the log discrepancy of $\Gamma_{i}$
with respect to $(X,D)$.

\begin{definition}
Let $(X,D)$ be a pair. We say that $(X,D)$ is a klt (resp. lc) pair
if there exists a log resolution $f:Y\rightarrow X$ as above such
that $a_{i}<1$ (resp. $a_{i}\leq1$).
\end{definition}

The abbreviation klt stands for Kawamata log terminal and lc for log
canonical. In this paper we will mainly deal with pairs $(X,D)$
where $X$ is a smooth variety. In this case the discrepancies
measure how singular is the effective divisor $D$. In case $D$ has
simple normal crossing support then we get only a restriction on the
coefficients of $D$. More precisely we have the following result,
see Corollary 3.12 in \cite{Kol} for a proof.

\begin{proposition}\label{klt}
Let $X$ be a smooth variety and $D=\sum d_{i}D_{i}$ an effective
divisor with simple normal crossing support. Then $(X,D)$ is klt
(resp. lc) if and only if $d_{i}<1$ (resp. $d_{i}\leq1$) for all i.
\end{proposition}

Roughly speaking the goal of the minimal model program is to find
the simplest variety birationally equivalent to a given variety. The
natural way to measure how simple a variety can be, is to look at
its canonical divisor and its intersection with curves on the
variety. So it is fundamental to understand $(K_{X})_{\leq 0}$
inside $\overline{NE}(X)$. The celebrated Cone Theorem describe how
the negative part looks like: it is generated by countably many
rational curves with bounded intersection with $K_{X}$. For details
see \cite{KM}.

\begin{theorem}[Cone theorem]\label{cone}
Let $(X,D)$ be a lc pair. Then there are countably many rational
curves $C_{j}\subset X$ such that $0<-(K_{X}+D)\cdot C_{j}\leq 2n$
and \bdism \overline{NE}(X)=\overline{NE}(X)_{(K_{X}+D)\geq
0}+\sum\rr_{\geq0}[C_{j}]. \edism If $X$ is smooth and $D=\emptyset$
then we can choose $C_{j}$ such that $0<-K_{X}\cdot C_{j}\leq n+1$.
\end{theorem}
In \cite{KM} the cone theorem is proved for klt pairs. The case of
lc pairs is treated in \cite{Ambro} and \cite{Fujino}.

It will be important for us to understand if certain adjoint linear
systems are semi-ample. Our main tool to study semi-ampleness is the
base point free theorem. It is one of the main step in the proof of
the cone theorem and it is of great interest on its own. Recall that
a divisor $L$ is big if it has maximal Kodaira dimension and it is
nef if $L\cdot C\geq 0$ for any irreducible curve $C$.

\begin{theorem}[Base point free theorem]\label{bpf}
Let $(X,D)$ a klt pair and let $L$ be a nef Cartier divisor such
that $aL-(K_{X}+D)$ is big and nef for some $a>0$. Then the linear
system $|mL|$ is base point free for any $m\gg 0$.
\end{theorem}

See \cite{KM} for the proof of the base point free theorem and more
details.

A divisor is said to be strictly nef if $L\cdot C > 0$ for any
irreducible curve $C$. Note that a strictly nef divisor is not
necessary ample even if it is big, see Section 5 and \cite{Hart2}
for examples. On the other hand the base point free theorem implies
the following.

\begin{corollary}\label{strictly}
Let $(X,D)$ be a klt pair such that $K_{X}+D$ is big and strictly
nef. Then $K_{X}+D$ is ample.
\end{corollary}

\begin{proof}

Since $K_{X}+D$ is big and nef by Theorem \ref{bpf} it is
semi-ample. Then for $m$ big enough $m(K_{X}+D)$ is base point free
and it defines a morphism $\phi_{|m(K_{X}+D)|}:X\rightarrow
Y\subseteq \pp^{k}$, where $k=\dim H^{0}(X,\oo_{X}(m(K_{X}+D)))-1$.
Suppose $C$ is a curve in some fiber of $\phi$ then $(K_{X}+D)\cdot
C=0$, but this is impossible since $K_{X}+D$ is strictly nef. Then
$\phi$ is a finite morphism and Corollary 1.2.15 in \cite{Laz1}
implies that $K_{X}+D$ is ample.
\end{proof}
Note that the same proof tells us that strictly nef and semi-ample
divisors are ample.

The base point free theorem is one of the first step toward a
solution of the following conjecture.
\begin{conjecture}[Abundance] \label{abundance}
Let $(X,D)$ a lc pair such that $K_{X}+D$ is nef. Then $K_{X}+D$ is
semi-ample.
\end{conjecture}
The abundance conjecture in such generality it is known only if
$\dim(X)\leq 3$, see \cite{Kol2}, \cite{Kee} and \cite{Kee2}.
Furthermore the base point free theorem implies that if $(X,D)$ is a
klt pair with $K_{X}+D$ big and nef then $K_{X}+D$ is semi-ample.

Note that Conjecture \ref{abundance} is not a trivial statement even
for surfaces. If the abundance conjecture holds then it implies that
if $K_{X}+D$ is strictly nef then it is ample.

The difference between abundance for klt pairs and abundance for lc
pairs is very subtle. The base point free theorem, as stated in
Theorem \ref{bpf}, it is not know for lc pairs. Actually Theorem
\ref{bpf} for lc pairs in dimension $n+1$ implies the abundance
conjecture for klt pairs of any Kodaira dimension in dimension $n$.
In particular it implies abundance with $X$ smooth and
$D=\emptyset$. See Theorem A.6 in \cite{Lazic} for details.

In the study of K\"ahler-Einstein metrics is fundamental to
understand where the map associated to some power of a big and nef
line bundle defines and embedding into some projective space. This
locus can be understood using a refinement of the stable base locus.
We recall here the key definition and we refer to Definition 10.3.2
in \cite{Laz2} and the discussion there for more details.
\begin{definition}
The stable base locus of a divisor $L$ is \bdism
\B(L):=\bigcap_{m\geq 1}\bs(|mL|), \edism
where $\bs(|L|)$ is the base locus of $L$. \\
The augmented base locus of a divisor $L$ is the Zariski-closed set
\bdism \B_{+}(L):=\B(L-\epsilon A), \edism for any ample $A$ and
sufficiently small $\epsilon>0$.
\end{definition}

Note that $\B_{+}(L)$ is a closed subset of $X$ if and only if $L$
is big and $\B_{+}(L)=\emptyset$ if and only if $L$ is ample. In
Section \ref{Modulo} we will see how $\B_{+}(L)$ is related to the
locus where the map $\phi_{|mL|}$ is not an embedding, for $m$
sufficiently large.

Quite remarkably the augmented base locus of a big and nef divisor
can be described numerically thanks to a theorem of Nakamaye. See
Theorem 10.3.5 in \cite{Laz2}.

\begin{theorem}\label{Nakamaye}
Let $X$ be a smooth projective variety and let $L$ be a big and nef
divisor on $X$. Then $\B_{+}(L)$ is the union of all positive
dimensional subvarieties $V\subseteq X$ such that $L^{\dim(V)}\cdot
V=0$.
\end{theorem}

Note that if a subvariety $V\subseteq X$ is such that
$L^{\dim(V)}\cdot V=0$ then $V\subseteq \B_{+}(L)$, so the content
of the theorem is to prove the converse.

\section{Positivity for varieties of general and log-general type}\label{General}

As recalled in Section \ref{Kahler-Einstein}, given a pair $(X, D)$
for which $K_{X}+D$ is ample then there exists a K\"ahler-Einstein
metric on $X\backslash D$ with Poincar\'e singularities along $D$.
Following intuitions of G. Tian \cite{Tian}, this K\"ahler-Einstein
metric should arise as the limit of a $1$-parameter family of
K\"ahler-Einstein metrics with cone-edge singularities along $D$. It
would then be desirable to have a uniform interval of angles,
depending on the dimension only, for which the existence of such
metrics is unobstructed, see Proposition \ref{luca}. Thus, to warm
up, we study the following.

\begin{question}\label{q1}
Let $X^{n}$ be a smooth variety and let $D$ be a reduced effective
divisor such that $K_{X}+D$ is ample. Is there a fixed
$\alpha_{n}<1$, which depends only $n$, such that $K_{X}+\alpha D$
is ample for any $\alpha\in\left(\alpha_{n},1\right]$?
\end{question}

In order to deal with the previous problem we need to study the nef
threshold.

\begin{definition}\label{threshold}
Let $X$ be a smooth projective variety and let $D$ be an integral
divisor such that $K_{X}+tD$ is nef for some $t\in\rr_{\geq 0}$. The
nef threshold is \bdism r(X,D):=\inf\left\{ t \in \rr_{\geq 0} \:|\:
K_{X}+t D \:\text{is nef}\right\}. \edism
\end{definition}

Since ampleness is an open condition, see Proposition 1.3.7 in
\cite{Laz1}, we know that $r(X,D)<1$ in Question \ref{q1}. The
content of the question is then if it is possible to bound the nef
thresholds away from one uniformly for all varieties of fixed
dimension. In order to obtain a uniform bound we use the following
well known corollary of the cone theorem.

\begin{lemma}\label{lem1}
Let $X$ be a smooth projective variety with $\dim(X)=n$ and let $H$
be an ample divisor. Then $K_{X}+(n+1)H$ is nef.
\end{lemma}

\begin{proof}
By Theorem \ref{cone} we know that any curve $C$ in $X$ is
numerically equivalent to \bdism a_{1}C_{1}+\dots+a_{r}C_{r}+F,
\edism where $a_{i}$ are positive real numbers, $C_{i}$ are rational
curves such that $0>C_{i}\cdot K_{X}\geq -(n+1)$ and $K_{X}\cdot
F\geq 0$. Since $H$ is ample, we have $H\cdot C_{i}\geq 1$. Then
$(K_{X}+(n+1)H)\cdot C\geq 0$.
\end{proof}

The proposition below gives the desired uniform bound for the nef
threshold when $K_{X}+D$ is ample.

\begin{proposition}\label{nef1}
Let $X$ be a smooth projective variety of dimension $n$ and let $D$
be a divisor such that $K_{X}+D$ is ample. Then \bdism r(X,D)\leq
\frac{n+1}{n+2}. \edism
\end{proposition}

\begin{proof}
Since $K_{X}+D$ is ample by Lemma \ref{lem1} we know that
$K_{X}+(n+1)(K_{X}+D)$ is nef. In particular $(n+2)K_{X}+(n+1)D$ is
nef and then the result follows.
\end{proof}

Let us observe that the bound given in Proposition \ref{nef1} is
\emph{sharp}.

\begin{remark}
Take $X=\pp^{n}$ and $D=(n+2)H$ where $H$ is the hyperplane section.
Then $K_{\pp^{n}}+D$ is ample and $r(\pp^{n},D)=\frac{n+1}{n+2}$.
\end{remark}

The following corollary provides a satisfactory answer to Question
\ref{q1}.

\begin{corollary}
Let $X$ and $D$ as in the Proposition \ref{nef1}. Then for any
$\alpha\in(\frac{n+1}{n+2}, 1]$ the divisor $K_{X}+\alpha D$ is
ample.
\end{corollary}

\begin{proof}
The result follows from the fact that in this case the nef threshold
can be obtained as $r(X,D)=\inf\left\{ t \in \rr_{\geq 0} \:|\:
K_{X}+t D \:\text{is ample}\right\}$.
\end{proof}

In the above discussion, we restricted ourselves to pairs $(X, D)$
with $K_{X}+D$ ample. Nevertheless, it is easy to construct examples
of pairs such that $X\backslash D$ admit complete negatively curved
K\"ahler-Einstein metrics of finite volume, but for which $K_{X}+D$
is \emph{not} ample. A wealth of examples is provided by the theory
of locally symmetric varieties and their compactifications, see for
example \cite{Ash}. More concretely, let us consider the following
class of low-dimensional examples. For more details the interested
reader may also refer to Section 4 in \cite{Luca1}.

\begin{example}\label{toroidal}
Let $X^{2}$ be a \emph{smooth} toroidal compactification of a
finite-volume complex-hyperbolic surface and let $D$ be the
compactifying divisor. The set $D$ simply consists of smooth
disjoint elliptic curves. Note that $X\backslash D$ admits a natural
finite-volume K\"ahler-Einstein metric coming from the Bergman
metric on $\cc\mathcal{H}^{2}$. Nevertheless, the line bundle
$K_{X}+D$ is just big and nef since $(K_{X}+D)\cdot D=0$.
\end{example}

The particular class of toroidal compactifications considered in
Example \ref{toroidal} has other remarkable properties. For a proof
of the following fact see Theorem 3.4 in \cite{LucaE}.

\begin{fact}
Let $(X, D)$ be as in Example \ref{toroidal}. Then for any $\alpha$
close enough the divisor $K_{X}+\alpha D$ is ample.
\end{fact}

It is therefore natural to address the following set of problems.

\begin{question}\label{q14}
Let $X^{n}$ be a smooth variety and let $D$ be a reduced effective
divisor such that $K_{X}+D$ is big and nef. Do we always have $r(X,
D)<1$? Moreover, if $r(X, D)<1$, can we find $\alpha_{n}<1$
depending only on $n$ such that $K_{X}+\alpha D$ is big and nef for
any $\alpha\in\left(\alpha_{n},1\right]$? Finally, can we
characterize the pairs $(X, D)$ for which $K_{X}+\alpha D$ is ample
for $\alpha$ close enough to one?
\end{question}

When dealing with these kind of questions is probably best to start
analyzing a weak notion of positivity: bigness. Thus, analogously to
the nef threshold, we define the pseudo-effective threshold.

\begin{definition}
Let $X$ be a smooth projective variety and let $D$ be an integral
divisor such that $K_{X}+tD$ is big for some $t\in\rr_{\geq 0}$. The
pseudo-effective threshold is \bdism \tau(X,D):=\inf\left\{ t \in
\rr_{\geq 0} \:|\: K_{X}+t D \:\text{is pseudo-effective}\right\}.
\edism
\end{definition}

Recall that the big cone is open and its closure is the cone of
pseudo-effective divisors, see Theorem 2.2.26 in \cite{Laz1}. Then
if $K_{X}+D$ is big we immediately obtain that $\tau(X,D)<1$. Since
nef divisors are also pseudo-effective we have that $\tau(X,D)\leq
r(X,D)$. Let us construct an example for which these numerical
invariants are actually different, compare with Question \ref{q14}.

\begin{example}\label{gabriele}
Let $X$ a smooth projective surface and $D$ a smooth irreducible
divisor such that $K_{X}+D$ is ample. Let $\pi:Y\rightarrow X$ be
the blow up of a point in $D$. Let $D'$ be the strict transform of
$D$ on $Y$. Then $K_{Y}+D'$ is big and nef but $K_{Y}+\alpha D'$ is
not nef for any $\alpha<1$. Then $r(Y,D')=1$ but $\tau(Y,D')<1$.
\end{example}

The phenomena explained in Example \ref{gabriele} relies on the fact
that, differently from the cone of big divisor, the cone of nef
divisors is \emph{closed} being the closure of the ample cone, see
Theorem 1.4.23 in \cite{Laz1}. Thus, we do not expect the nefness to
hold when we decrease the coefficients of $D$. Remarkably, the
bigness of $K_{Y}+D^{'}$ in Example \ref{gabriele} does not help in
achieving $r(Y, D^{'})<1$.

We are now interested in trying to characterize the pairs $(X, D)$
with $K_{X}+D$ nef and with nef threshold strictly less that one.
The following result gives a quite satisfactory answer, compare with
Question \ref{q14}.

\begin{proposition}\label{prop}
Let $X$ be a smooth projective variety and let $D$ be a divisor such
that $K_{X}+D$ is nef. Then $r(X,D)<1$ if and only if there are no
irreducible curves $C$ such that $(K_{X}+D)\cdot C=0$ and
$K_{X}\cdot C<0$.
\end{proposition}

\begin{proof}
Suppose $r(X,D)<1$. Let $C$ be an irreducible curve such
$(K_{X}+D)\cdot C=0$. In particular we have that $D\cdot
C=-K_{X}\cdot C$. We want to prove that $K_{X}\cdot C\geq 0$. Let
$\alpha$ be close to one such that $K_{X}+\alpha D$ is nef. Then
\bdism (K_{X}+\alpha D)\cdot C=(1-\alpha)K_{X}\cdot C\geq 0. \edism
Conversely suppose that there are no irreducible curves $C$ with the
properties $(K_{X}+D)\cdot C=0$ and $K_{X}\cdot C<0$. We would like
to show that $r(X,D)<1$. Let $C'$ be an irreducible curve such that
$(K_{X}+D)\cdot C'>0$. By the same argument in the proof of Lemma
\ref{lem1} we can write $C'\equiv a_{1}C_{1}+\dots +a_{r}C_{r}+F$
where $a_{i}$ are positive numbers, $C_{i}$ are irreducible curves
with $0>K_{X}\cdot C_{i}\geq -(n+1)$ and $K_{X}\cdot F\geq 0$. Since
$F$ is a limit of effective 1-cycles and $K_{X}+D$ is nef we have
that $(K_{X}+D)\cdot F\geq 0$. By assumption if $(K_{X}+D)\cdot
C_{i}=0$ then $K_{X}\cdot C_{i}\geq 0$. In particular
$\left(K_{X}+(n+1)(K_{X}+D)\right)\cdot C'\geq 0$. This implies that
$(K_{X}+\frac{n+1}{n+2}D)\cdot C'\geq 0$ for any curve $C'$ such
that $(K_{X}+D)\cdot C'>0$. Then we need only to check the nefness
of $K_{X}+\alpha D$ on irreducible curve $C$ such that
$(K_{X}+D)\cdot C=0$. The same computation as before tells us that
$(K_{X}+\alpha D)\cdot C=(1-\alpha)K_{X}\cdot C\geq 0$.
\end{proof}

Let us observe that Proposition \ref{prop} implies the following.

\begin{corollary}\label{gap}
Let $X$ be a smooth projective variety and let $D$ be a divisor such
that $K_{X}+D$ is nef. Then either $r(X, D)=1$ or $r(X,
D)\leq\frac{n+1}{n+2}$.
\end{corollary}

The reasoning given in Proposition \ref{prop} is easily adapted to
understand the obstructions for $K_{X}+\alpha D$ to be
\emph{strictly} nef. In fact, we can state:

\begin{proposition}\label{prop2}
Let $X$ be a smooth projective variety and let $D$ be a divisor such
that $K_{X}+D$ is nef. Then $K_{X}+\alpha D$ is strictly nef for
$\alpha\in(\frac{n+1}{n+2}, 1)$ if and only if there are no
irreducible curves $C$ such that $(K_{X}+D)\cdot C=0$ and
$K_{X}\cdot C\leq0$.
\end{proposition}

We are now ready to address the last problem considered in Question
\ref{q14}.

\begin{theorem}\label{tian}
Let $X$ be a smooth projective variety and let $D$ be a reduced
effective divisor with simple normal crossing support such that
$K_{X}+D$ is big and nef. Then $K_{X}+\alpha D$ is ample for
$\alpha\in(\frac{n+1}{n+2}, 1)$ if and only if there are no
irreducible curves $C$ such that $(K_{X}+D)\cdot C=0$ and
$K_{X}\cdot C\leq0$.
\end{theorem}

\begin{proof}
Suppose $K_{X}+\alpha D$ to be ample. Let $C$ be an irreducible
curve such that $(K_{X}+D)\cdot C=0$. Then
\begin{align}\notag
(K_{X}+\alpha D)\cdot C=(1-\alpha)K_{X}\cdot C>0.
\end{align}
Conversely, by Proposition \ref{prop2} we know that $K_{X}+\alpha D$
is strictly nef for $\alpha\in (\frac{n+1}{n+2}, 1)$. Since $\tau(X,
D)\leq r(X,D)$ we conclude that $K_{X}+\alpha D$ is big and strictly
nef for $\alpha\in (\frac{n+1}{n+2}, 1)$. Note that $(X,\alpha D)$
is a klt pair. By Theorem \ref{bpf} we know that $K_{X}+\alpha D$ is
semi-ample. Since strictly nef semi-ample divisors are ample the
proof is complete.
\end{proof}

We can now give a proof of Theorem \ref{th1} stated in the
Introduction \ref{introduction}.

\begin{proof}[Proof of Theorem \ref{th1}]
Combine Proposition \ref{luca} and Theorem \ref{tian} with the
analytical results of Campana-Guenancia-P\u aun \cite{Campana1} and
Mazzeo-Rubinstein \cite{Mazzeo2}, more precisely Theorem A in
\cite{Campana1} and Theorem 1.3 in \cite{Mazzeo2}.
\end{proof}

Theorem \ref{tian} nicely applies in low-dimensions. Thus, let us
collect few corollaries concerning complex surfaces.

\begin{corollary}\label{cor2}
Let $X$ be a smooth projective surface and let $D$ be a divisor such
that $K_{X}+D$ is big and nef. Suppose there are no curves $C$ in
$X$ such that $C=\pp^{1}$, $C^{2}=-1$ and $C\cdot D=1$, then
$r(X,D)<1$.
\end{corollary}

\begin{proof}
We want to apply Proposition \ref{prop}. Let $C$ be a curve such
that $(K_{X}+D)\cdot C=0$. Then by the Hodge index theorem
$C^{2}<0$. By adjunction we have \bdism K_{X}\cdot C=2g-2-C^{2}.
\edism Suppose $K_{X}\cdot C<0$. Since $C^{2}<0$ this can happen if
and only if $g=0$ and $C^{2}=-1$. But then $D\cdot C=-K_{X}\cdot
C=1$ which contradicts our hypothesis.
\end{proof}

Similarly, we have a nice characterization of the pairs with ample
twisted log-canonical bundles.

\begin{corollary}\label{cor3}
Let $X$ be a smooth projective surface and let $D$ be a reduced
effective divisor with simple normal crossing support such that
$K_{X}+D$ is big and nef. Suppose there are no curves $C$ in $X$
such that $C=\pp^{1}$, $C^{2}=-1$ and $C\cdot D=1$, or $C=\pp^{1}$,
$C^{2}=-2$ and $C\cdot D=0$. Then $K_{X}+\alpha D$ is ample for any
$\alpha\in(\frac{3}{4}, 1)$.
\end{corollary}

\begin{proof}
We want to apply Theorem \ref{tian}. Because of Corollary
\ref{cor2}, it suffices to characterize the curves $C$ such that
$(K_{X}+D)\cdot C=0$ and $K_{X}\cdot C=0$. By the Hodge index
theorem we have $C^{2}<0$ and then by adjunction $K_{X}\cdot C=0$ if
and only if $g=0$ and $C^{2}=-2$. This implies $C\cdot D=0$.
\end{proof}

The interested reader should compare Corollaries \ref{cor2} and
\ref{cor3} with Theorems 3.3 and 3.4 in \cite{LucaE}. The results
presented in \cite{LucaE} contain a geometric characterization of
the pairs $(X^{2}, D)$ with $K_{X}+D$ big and nef. Such results rely
on F. Sakai's notions of semi-stability and $D$-minimality, see
\cite{Sakai}.

It is now interesting to collect a corollary regarding Fano
varieties. Recall that $X$ is called a Fano variety if $-K_{X}$ is
ample.

\begin{corollary}
Let $X$ be a smooth Fano variety. Let $D$ be an effective
divisor such that $K_{X}+D$ is nef. Then
$r(X,D)<1$ if and only if $K_{X}+D$ is ample.
\end{corollary}

\begin{proof}
If $K_{X}+D$ is ample then it is clear that $r(X,D)<1$. Conversely, assume that  $r(X,D)<1$. 
Since $X$ is Fano, we have $K_{X}\cdot C<0$ for any irreducible curve $C$.
By Proposition \ref{prop2} we conclude that $K_{X}+D$ is strictly nef. 
If we now apply Theorem \ref{bpf} to $X$ with $L=K_{X}+D$ and $a=2$, we conclude that a multiple of $L$, say $mL$, is semi-ample. 
Since we have already shown that $L$ is strictly nef we obtain that $L$ is indeed ample.  
\end{proof}

In the discussion above, we mainly focused on the study of the nef
threshold. We now would like to better understand the behavior of
the pseudo-effective threshold. First, if $r(X, D)<1$ then $\tau(X,
D)$ is automatically bounded away from one. It remains to understand
if the same is true in the case $r(X, D)=1$. The fact that $K_{X}+D$
is assumed to be nef turns out to be very useful. In fact, recall
the following theorem of Andreatta, see Theorem 5.1 in \cite{And}.

\begin{theorem}[Andreatta]\label{and}
Let $X$ be a smooth projective variety of dimension $n$ and let $L$
be a big and nef divisor. Then $K_{X}+(n+1)L$ is pseudo-effective.
\end{theorem}

It is interesting to notice how Theorem \ref{and} plays a similar
role as Lemma \ref{lem1}. Thus, the idea of the proof of Proposition
\ref{nef1} gives the following.

\begin{proposition}\label{big1}
Let $X$ be a smooth projective variety and let $D$ be a divisor such
that $K_{X}+D$ is big and nef. Then \bdism
\tau(X,D)\leq\frac{n+1}{n+2}. \edism
\end{proposition}

As expected, we have:

\begin{corollary}
Let $X$ and $D$ as in the Proposition \ref{big1}. Then for any
$\alpha>\frac{n+1}{n+2}$ the divisor $K_{X}+\alpha D$ is big.
\end{corollary}

Let us now briefly discuss another approach to the problem addressed
in Proposition \ref{big1}. Recall the following result of Koll\'ar
\cite{Kol}.

\begin{theorem}[Koll\'ar]\label{Kollar}
Let $X$ be a smooth variety and let $L$ be a big and nef divisor.
Then $K_{X}+mL$ gives a birational map for any
$m\geq\binom{n+2}{2}$.
\end{theorem}

Note that, given $L$ as in \ref{Kollar}, then
$K_{X}+\binom{n+2}{2}L$ is not only pseudo-effective but also big.
Thus, by letting $L=K_{X}+D$ and proceeding as in Proposition
\ref{big1} we obtain $\tau(X, D)\leq
\frac{\binom{n+2}{2}}{1+\binom{n+2}{2}}$. Despite the fact that this
new bound is worse than the one given in Proposition \ref{big1},
this new approach, based on Theorem \ref{Kollar}, provides the idea
of how to extend the result when $K_{X}+D$ is only big. In fact, it
strongly suggests that results in effective birationality can help
in bounding the pseudo-effective threshold. Recently Hacon,
M$^{\text{c}}$Kernan and Xu generalized Theorem \ref{Kollar} in
\cite{HMX} and they proved the following.

\begin{theorem}[Hacon-M$^{\text{c}}$Kernan-Xu]\label{hmx1}
Let $X$ be a smooth projective variety and let $D$ be a reduced
effective divisor with simple normal crossing such that $K_{X}+D$ is
big. Then there exists a positive number $\alpha_{n}$ depending only
on the dimension of $X$ such that $K_{X}+\alpha D$ is big for all
$\alpha\in (\alpha_{n}, 1]$.
\end{theorem}

Actually they proved a more general statement when $(X,D)$ is a lc
pair and the coefficients of $D$ are in a fixed DCC set $I$.\\

In the rest of the section we address the following:

\begin{question}\label{q2}
Let $X$ be a smooth variety with $K_{X}$ big and nef and let $D$ be
an effective divisor. Can we characterize the pairs $(X, D)$ with
$K_{X}+\alpha D$ big and nef or ample for $\alpha$ close enough to
zero? Finally, is there a fixed $\alpha_{n}>0$, which depends only
on $n$, such that $K_{X}+\alpha D$ is big and nef or ample for any
$\alpha\in\left[0,\alpha_{n}\right)$?
\end{question}

Note that in this case $K_{X}+\alpha D$ is big for any $\alpha\geq
0$, so the main problem is to understand when it is nef. To this aim
we use again the cone theorem, but now we need $(X,D)$ to be a lc
pair. Thus, to warm up, let us observe the following.

\begin{proposition}
Let $(X,D)$ be a lc pair with $K_{X}$ ample. Then $K_{X}+\alpha D$
is ample for any \bdism \alpha < \frac{1}{2n+1}. \edism
\end{proposition}
\begin{proof}
By the cone theorem for lc pairs $K_{X}+D+2nK_{X}$ is nef. In
particular so is $K_{X}+\frac{1}{2n+1}D$.
\end{proof}

A similar approach can now be applied when $K_{X}$ is simply nef.

\begin{theorem}
Let $X$ be a smooth projective variety with $K_{X}$ nef. Let $D$ be
a reduced effective divisor with simple normal crossing support.
Then $K_{X}+\alpha D$ is nef for $\alpha\in [0, \frac{1}{2n+1}]$ if
and only if there are no irreducible curves $C$ such that
$(K_{X}+D)\cdot C<0$ and $K_{X}\cdot C=0$.
\end{theorem}

\begin{proof}
By the cone theorem for lc pairs, if there are no irreducible curves
$C$ such that $(K_{X}+D)\cdot C<0$ and $K_{X}\cdot C=0$ then the
$\rr$-divisor \bdism K_{X}+D+tK_{X} \edism is nef for $t\geq 2n$.
Conversely, if $K_{X}+\alpha D$ is nef then for some $\alpha\in(0,
1)$ then the curves that dot negatively with $K_{X}+D$ must
intersect positively with $K_{X}$.
\end{proof}

We conclude this section by studying when $K_{X}+\alpha D$ is ample
for $\alpha$ close to zero.

\begin{theorem}\label{tian2}
Let $X$ be a smooth projective variety with $K_{X}$ big and nef. Let
$D$ be a reduced effective divisor with simple normal crossing
support. Then $K_{X}+\alpha D$ is ample for $\alpha\in (0,
\frac{1}{2n+1})$ if and only if there are no irreducible curves $C$
such that $(K_{X}+D)\cdot C\leq 0$ and $K_{X}\cdot C=0$.
\end{theorem}

\begin{proof}
If there are no irreducible curves $C$ such that $(K_{X}+D)\cdot
C\leq 0$ and $K_{X}\cdot C=0$ we have that $K_{X}+\alpha D$ is
strictly nef for any $\alpha\in (0, \frac{1}{2n+1})$. Recall that
$K_{X}+\alpha D$ is a big $\rr$-divisor for any $\alpha\geq 0$. Then
Theorem \ref{bpf} implies that $K_{X}+\alpha D$ is semi-ample and
therefore ample. The converse should be clear by now.
\end{proof}

We do not expect the above result to be sharp. On the other hand we
expect that the previous result holds without the bigness assumption.

\begin{corollary}\label{abund}
Assume Conjecture \ref{abundance}. Let $X$ be a smooth projective variety
with $K_{X}$ nef. Let $D$ be a reduced effective divisor with
simple normal crossing support. Then $K_{X}+\alpha D$ is ample
for $\alpha\in (0, \frac{1}{2n+1})$ if and only if there are
no irreducible curves $C$ such that $(K_{X}+D)\cdot C\leq 0$
and $K_{X}\cdot C=0$.
\end{corollary}

\begin{proof}
In the proof of Theorem \ref{tian2}, replace Theorem \ref{bpf} with
Conjecture \ref{abundance}.
\end{proof}

Corollary \ref{abund} implies that Theorem \ref{th2} should hold
without asking $X$ to be of general type.

We conclude this section by giving the proof of Theorem \ref{th2}
stated in the Introduction \ref{introduction}.

\begin{proof}[Proof of Theorem \ref{th2}]
Combine Proposition \ref{luca} and Theorem \ref{tian2} with Theorem
1.3 in \cite{Mazzeo2}.
\end{proof}

\section{Remarks on ampleness for line bundles on quasi-projective varieties}\label{Modulo}

It is clear that the study of K\"ahler-Einstein metrics on
quasi-projective varieties must be connected with the notion of
\emph{ampleness} on open algebraic manifolds. It seems there is no
general agreement in the existing literature on the definition of
this concept. Recall that for compact manifold the notion of
ampleness is usually introduced algebraically and then formulated in
geometric and numerical terms, see Chapter 1 in \cite{Laz1}. Here,
we follow a completely analogous path. The interested reader should
compare the results contained in this section with some of the
existing literature \cite{Wu1}, \cite{Wu2}, \cite{TianY}.

Thus, let us start with a definition. For more details see
\cite{Takayama}.

\begin{definition}\label{takayama}
Let $X$ be a smooth projective variety and let $D$ be a simple
normal crossing divisor. A line bundle $L$ on $X$ is said to be
ample modulo $D$, if for any torsion free sheaf $\mathcal{F}$ on
$X$, there exists an integer $m_{0}$ such that $(\mathcal{F}\otimes
L^{n})_{|_{X\backslash D}}$ is generated by $H^{0}(X,
\mathcal{F}\otimes L^{n})$ for $m\geq m_{0}$.
\end{definition}

We then say that a divisor $L$ is very ample modulo $D$ if the
rational map $\phi_{|L|}$ associated to the linear system $|L|$
defines an embedding of $X\backslash D$ into some projective space.
By the same argument of Theorem II.7.6 in \cite{Har}, one can show
that some power of an ample modulo $D$ divisor is very ample modulo
$D$. In particular, it follows that if $L$ is ample modulo $D$ then
it is necessarily big. Let us continue with some general
observations. First, note that $\textbf{B}_{+}(L)$ is a proper
subset of $X$ since $L$ is big. Next, recall that $\B_{+}(L)$ can be
written as
\begin{align}\notag
\textbf{B}_{+}(L)=\bigcap_{L=A+E}\supp(E),
\end{align}
where the intersection is taken over all decompositions $L=A+E$,
where $A$ is ample and $E$ effective, see Remark 1.3 in \cite{Ein}.
Since $X$ is noetherian and $\B_{+}(L)$ is a Zariski closed subset
of $X$ we can find finitely many decompositions $L=A_{i}+E_{i}$ such
that $\textbf{B}_{+}(L)=\cap_{i=1}^{k}\supp(E_{i})$. In particular,
this implies that for $m$ large enough, the linear system $|mL|$
defines an embedding of $X\backslash \B_{+}(L)$ into some projective
space. This simple argument gives the following.

\begin{fact}\label{inclusione}
If $L$ is ample modulo $D$ then $\B_{+}(L)\subseteq D$.
\end{fact}

\begin{remark}
It is interesting to observe that, in a recent preprint \cite{BCL}, Boucksom-Cacciola-Lopez proved that 
$\B_{+}(L)$ is the smallest closed subset $V$ of $X$ such that the linear system $|mL|$
defines an isomorphism of $X\backslash V$ onto its image for sufficiently large $m$. In particular,  
$L$ is ample modulo $D$ if and only if $\B_{+}(L)\subseteq D$.
\end{remark}

It would now be desirable to have a numerical characterization of
ampleness modulo $D$.

\begin{definition}
Let $X$ be a smooth projective variety and let $D$ be an effective
divisor. A line bundle $L$ on $X$ is said to be strictly nef modulo
$D$, if $L\cdot C>0$ for any curve $C$ not entirely contained in
$D$.
\end{definition}

Let us observe that, similarly to the compact case, being strictly
nef modulo $D$ is not enough for achieving ampleness modulo $D$. To
this aim recall that a strictly nef divisor is not necessarily
ample. The first construction of such a divisor is due to Mumford.
In fact, he showed the existence of a surface $X$ with a strictly
nef divisor $D$ such that $D^{2}=0$. Ramanujam, based on Mumford's
example, showed the existence of a big and strictly nef divisor that
is not ample. Note that by Nakai's criterion a divisor which is big
and strictly nef but not ample can exist in dimension greater or
equal then three only. For the convenience of the reader we sketch
the construction of these examples. For details see Chapter I $\S
10$ in \cite{Hart2}.

\begin{example}[Mumford, Ramanujam]
We start with Mumford's example. Let $C$ be a smooth curve of genus
$g\geq 2$. By a theorem of Seshadri there exists a stable vector
bundle $E$ of degree zero over $C$ of rank two such that all its
symmetric powers are stable. Let $X=\pp(E)$ and $D=\oo_{X}(1)$ be
the Serre line bundle. One can check that $D$ is strictly nef. Since
$E$ has degree zero, $D^{2}=0$ and in particular $D$ is not ample.
Now we can give Ramanujam's example of a big and strictly nef
divisor which is not ample. We will show the existence of such
divisor using Mumford's example. Let $X$ be as above and let $H$ be
an effective ample divisor on $X$. Define $Y:=\pp(\oo_{X}(D-H)\oplus
\oo_{X})$ and let $\pi:Y\rightarrow X$ be the projection. Let
$X_{0}\subseteq Y$ be the zero section. Finally let
$L:=X_{0}+\pi^{*}H$. As shown in \cite{Hart2}, $L$ is strictly nef
but it is not ample because $L^{2}\cdot X_{0}=D^{2}=0$. Furthermore
$L$ is big because $L^{3}=(D+H)\cdot H>0$.
\end{example}

Thus, let $L$ be a big and strictly nef divisor which is not ample.
Observe that $\B_{+}(L)$ is not empty otherwise $L$ would be ample.
Then $L$ is strictly nef modulo $D$ for any effective divisor $D$.
Nevertheless, by choosing a divisor $D$ that does not contain
$\B_{+}(L)$ we certainly obtain that $L$ cannot be ample modulo $D$,
see Fact \ref{inclusione}.

The discussion above tells us that if we want to control $\B_{+}(L)$
only with intersection numbers we have to take in account all
positive dimensional subvarieties and not only curves. In fact, let
$L$ be \emph{nef} and such that $L^{k}\cdot V>0$ for any
$k$-dimensional subvariety not entirely contained in $D$. Then by
Theorem \ref{Nakamaye} we have that $\B_{+}(L)$ must be contained in
$D$. This is then enough to conclude:

\begin{fact}
Let $L$ be nef such that $L^{k}\cdot V>0$ for any $k$-dimensional
subvarieties not entirely contained in $D$. Then $L$ is ample modulo
$D$.
\end{fact}

Let us now study line bundles which are semi-ample. We would like to
show that if $L$ is semi-ample and strictly nef modulo $D$ then it
must be ample modulo $D$. First, since $L$ is semi-ample then there
exists a positive integer $m$ such that the linear system $|mL|$
defines a morphism into some projective space. Let us denote this
map by $\phi_{|mL|}$. Since $L$ is strictly nef modulo $D$ we have
that $\phi_{|mL|}$ is a finite map over $X\backslash D$. This
immediately implies that $L$ is big. Moreover, a simple integration
shows that $L^{k}\cdot V>0$ for any $k$-dimensional subvariety which
is not entirely contained in $D$. By Theorem \ref{Nakamaye} we
conclude that $\B_{+}(L)$ is contained in $D$. It then follows that
$L$ is ample modulo $D$.

\begin{fact}\label{finite}
Let $L$ be semi-ample and strictly nef modulo $D$. Then $L$ is ample
modulo $D$.
\end{fact}

Let us construct examples of line bundles which are ample modulo a
divisor. A wealth of examples can be constructed using the
following.

\begin{fact}
Let $X$ be a smooth projective variety and let $D$ be a simple
normal crossing divisor such that $K_{X}+\alpha D$ is ample for some
$\alpha\in[0, 1)$. Then $K_{X}+D$ is ample modulo $D$.
\end{fact}

We now want to give examples of pairs where $K_{X}+D$ is ample
modulo $D$ but such that $K_{X}+\alpha D$ is never ample for any
$\alpha\in \rr$.

\begin{example}\label{menodue}
Let $X$ be a smooth projective surface and let $D$ be a semi-stable
curve. Let assume $(X, D)$ to be $D$-minimal, of log-general type,
without interior $(-2)$-curves. Moreover, assume the existence of at
least one $(-2)$-curve in $D$, denoted by $E$, such that
$E\cdot(D-E)=2$. Note that $K_{X}+D$ is semi-ample since the
abundance conjecture holds in dimension two. Moreover, $K_{X}+D$ is
strictly nef modulo $D$ but it is not ample since $(K_{X}+D)\cdot
E=0$. By Fact \ref{finite} we conclude that $K_{X}+D$ is ample
modulo $D$. On the other hand $(K_{X}+\alpha D)\cdot E=0$ constantly
in $\alpha$.
\end{example}

For more details regarding Example \ref{menodue} we refer to Section
3 in \cite{LucaE}. If the reader is interested in explicitly
constructing a pair as in Example \ref{menodue} we refer to Example
\ref{Damin} below.

Now, in all the above examples $r(X, D)<1$. Let us construct a pair
$(X, D)$ with $K_{X}+D$ big, semi-ample and strictly nef modulo $D$
such that $r(X, D)=1$.

\begin{example}\label{lucagab}
Let $X$ a smooth projective surface and $D$ a reduced divisor
consisting of two components $D_{1}$ and $D_{2}$ meeting
transversally at a point $p$. Assume $K_{X}+D$ to be ample. Let
$\pi:Y\rightarrow X$ be the blow up at $p\in D$. Let $D'$ be the
union of the strict transforms of $D_{1}$ and $D_{2}$ in $Y$ and the
exceptional divisor $E$. Then $K_{Y}+D'=\pi^{*}(K_{X}+D)$. We
conclude that $K_{Y}+D'$ is big, nef and strictly nef modulo $D'$.
Then $K_{Y}+D'$ is ample modulo $D'$. On the other hand \bdism
(K_{Y}+\alpha D')\cdot E=(\alpha-1)<0 \edism which therefore implies
$r(X, D)=1$.
\end{example}

Unfortunately, as shown in Example \ref{lucagab}, being nef and
ample modulo a divisor is not an open condition even for divisors of
the form $K_{X}+D$. Nevertheless, we have the following.

\begin{theorem}\label{Amodulo}
Let $X$ be a smooth projective variety and let $D$ be an effective
reduced divisor such that $K_{X}+D$ is nef and strictly nef
modulo $D$. Assume that $r(X, D)<1$. Then $K_{X}+\alpha D$ is big,
nef and ample modulo $D$ for any $\alpha\in(\frac{n+1}{n+2}, 1)$.
\end{theorem}

\begin{proof}
Since we are assuming $r(X, D)<1$, by Corollary \ref{gap} we know
that
\begin{align}\notag
\tau(X, D)\leq r(X, D)\leq\frac{n+1}{n+2}.
\end{align}
It then remains to check the ampleness modulo $D$. As in the proof
of Proposition \ref{prop}, observe that given an irreducible curve
$C$ such that $(K_{X}+D)\cdot C>0$ then $(K_{X}+\alpha D)\cdot C>0$
for any $\alpha\in(\frac{n+1}{n+2}, 1]$. By Theorem \ref{bpf}, for
any $\alpha\in(\frac{n+1}{n+2}, 1)$ we have that $K_{X}+\alpha D$ is
semi-ample. By Fact \ref{finite} the proof is then complete.
\end{proof}

Note that the strategy behind the proof of Theorem \ref{Amodulo} can
be used to prove the following.

\begin{corollary}\label{semipositive}
Let $(X, D)$ be such that $K_{X}+D$ is big, nef and strictly nef
modulo $D$. Assume there are no irreducible curves $C$ such that
$(K_{X}+D)\cdot C=0$ and $K_{X}\cdot C<0$. Then, for any
$\alpha\in(\frac{n+1}{n+2}, 1)$, the $\rr$-cohomology class of
$K_{X}+\alpha D$ can be represented by a smooth closed $(1, 1)$-form
which is everywhere positive semi-definite and strictly positive
outside $D$.
\end{corollary}

\begin{proof}
By Proposition \ref{prop} we have that $r(X, D)<1$ if and only if
there are no irreducible curves $C$ such that $(K_{X}+D)\cdot C=0$
and $K_{X}\cdot C<0$.
\end{proof}

We decided to explicitly state Corollary \ref{semipositive} because
it has nice applications in the theory of negatively curved
K\"ahler-Einstein metrics on quasi-projective varieties. These
applications rely on some recent advancements in the theory of
degenerate complex Monge-Amp\`ere equations with singular right hand
side \cite{Pali}. Thus, let $(X, D)$ be as in Corollary
\ref{semipositive}. For any $\alpha\in(\frac{n+1}{n+2}, 1)$ denote
by $\gamma_{\alpha}$ a smooth semi-positive form representing the
cohomology class of $K_{X}+\alpha D$. Given a smooth volume form
$\Omega$ on $X$ and Hermitian metrics $\|\cdot\|_{i}$ on
$\mathcal{O}_{X}(D_{i})$, consider the family of degenerate complex
Monge-Amp\`ere equations
\begin{align}\label{degenerate}
(\gamma_{\alpha}+\sqrt{-1}\partial\overline{\partial}\varphi)^{n}=(\gamma_{\varphi}^{\alpha})^{n}=
e^{\varphi}\frac{\Omega}{\prod_i\|\sigma_{i}\|^{2\alpha}}
\end{align}
where by $D_{i}$ and $\sigma_{i}$ we respectively denote the
irreducible components of $D$ and the associated defining sections.
By Theorem 6.1 in \cite{Pali}, for a fixed $\alpha$ this equation
admits a unique solution $\varphi\in L^{\infty}(X)$ which is smooth
on $X\backslash D$. Moreover, by appropriately choosing $\Omega$ and
the $\|\cdot\|_{i}$ in \ref{degenerate} we can arrange
$\gamma^{\alpha}_{\varphi}$ to be Einstein with negative scalar
curvature on $X\backslash D$. To sum up, given a pair $(X, D)$ such
that $K_{X}+D$ is big, strictly nef modulo $D$ and with $r(X, D)<1$
then the K\"ahler-Einstein theory on $X\backslash D$ is rather
well-understood. Of course, the asymptotic behavior of
$\gamma^{\alpha}_{\varphi}$ along $D$ as well as the limiting
behavior of this family as $\alpha\rightarrow 1$ remain quite
mysterious.

\begin{remark}\label{Damin1}
The $1$-parameter family of negatively curved K\"ahler-Einstein
metrics $\gamma^{\alpha}_{\varphi}$ cannot represent ample
cohomology classes unless there are no irreducible curves $C$ such
that $(K_{X}+D)\cdot C=0$ and $K_{X}\cdot C\leq0$. Recall that in
this case the form $\gamma_{\alpha}$ can be chosen to be a K\"ahler
metric. Compare Proposition \ref{luca} and Theorem \ref{tian}. In
this case the asymptotic behavior of $\gamma^{\alpha}_{\varphi}$
along $D$ is much better understood. In fact, thanks to the recent
works \cite{Mazzeo} and \cite{Mazzeo2}, we know that for any
$\alpha\in(\frac{n+1}{n+2}, 1)$ the K\"ahler-Einstein current
$\gamma^{\alpha}_{\varphi}$ has edge singularities of cone angle
$2\pi(1-\alpha)$ along $D$.
\end{remark}

Let us conclude this section by giving an example of a pair $(X, D)$
which admits both complete and incomplete negatively curved
K\"ahler-Einstein metrics on $X\backslash D$, but such that
$K_{X}+\alpha D$ is never ample for any $\alpha\in\rr$.

\begin{remark}\label{Damin}
Let $\mathcal{H}\times\mathcal{H}$ be the product of two copies of
the upper half plane of $\cc$. Let $\Gamma$ be a discrete torsion
free non-uniform subgroup of $SL(2, \rr)\times SL(2, \rr)$. Denote
by $X^{*}$ the Baily-Borel compactification \cite{Ash} of
$(\mathcal{H}\times\mathcal{H})/\Gamma$. Let $X$ be the
\emph{minimal} resolution of $X^{*}$. By appropriately choosing
$\Gamma$, see \cite{Hirzebruch} for beautiful \emph{explicit}
constructions, the resolution of the cusp points in $X^{*}$ is given
by a divisor $D$ whose connected components are cycles of smooth
rational curves with negative self-intersection. More precisely, for
any of the irreducible components we have $D_{i}^{2}\leq-2$.
Furthermore, we have at least one $(-2)$-curve in $D$ while there
are no $(-2)$-curves in $X\backslash D$ since the curvature of
$\mathcal{H}\times\mathcal{H}$ is nonpositive. Now, let $D_{i}$ be a
$(-2)$-curve in $D$. We then have that $(K_{X}+\alpha D)\cdot
D_{i}=0$ constantly in $\alpha$. On the other hand $X\backslash D$
admits a \emph{complete} nonpositively curved K\"ahler-Einstein
metric coming from $\mathcal{H}\times\mathcal{H}$. Moreover, the
pair $(X, D)$ satisfies the hypothesis of Corollary
\ref{semipositive}, see Example \ref{menodue}. Thus, for
$\alpha\in(\frac{3}{4}, 1)$, $X\backslash D$ also admits a
$1$-parameter family of \emph{incomplete} negatively curved
K\"ahler-Einstein metrics $\gamma^{\alpha}_{\varphi}$. Of course,
this family has \emph{not} edge singularities along $D$ in the sense
of Definition \ref{Tian}.
\end{remark}

\section{Positivity for Fano and log-Fano varieties}\label{Fano}

In this section we study the existence of uniform bound for the nef
threshold when $X$ is a Fano variety or a log-Fano variety, i.e.
$-K_{X}$ or $-(K_{X}+D)$ is ample.

\begin{question}\label{q3}
Let $X$ be a smooth variety and let $D$ an effective divisor with
$-(K_{X}+D)$ ample. Is there a fixed $\alpha_{n}<1$, which depends
only on $n$, such that $-(K_{X}+\alpha D)$ is ample for any
$\alpha\in(\alpha_{n},1]$?
\end{question}

It is quite tempting to do the same thing as we did in the previous
section. For example if we apply Lemma \ref{lem1} we get that
$-(K_{X}+\frac{n+1}{n} D)$ is nef. Since we want to deal with
$\alpha<1$ we do not get any useful information. Probably the best
reason that explains why we cannot do the same thing is that the
above question has negative solution.

\begin{example}\label{ex1}
We follow the notation in \cite{Har}. Let $C=\pp^{1}$ and
$\mathcal{E}:=\oo_{C}\oplus\oo_{C}(-e)$ where $e$ is a positive
integer. We define $X_{e}:=\pp(\mathcal{E})$, it is a rational ruled
surface. By Corollary V.2.11 in \cite{Har} we know that
$K_{X_{e}}\equiv -2C_{0}-(e+2)f$ where $C_{0}$ is the zero section
and $f$ is a general fibre. Let $D:=C_{0}$. Then
$-K_{X_{e}}-D=C_{0}+(e+2)f$ is an ample divisor by Corollary V.2.18
in \cite{Har}. Then by the same corollary $-K_{X_{e}}-\alpha
D=(2-\alpha)C_{0}+(e+2)f$ is ample if and only if
$\alpha>\frac{e-2}{e}$. Then if $e$ goes to infinity $\alpha$ must
be arbitrarily close to one. Note that $(-(K_{X_{e}}+D))^{2}=e+4$.
\end{example}

We would like to understand how $\alpha$ depends on the pair
$(X,D)$. We will prove that the situation in Example \ref{ex1} is
the most general and that $\alpha$ depends only on the dimension of
$n$ and the top self-intersection of $-(K_{X}+D)$. We need two
preliminary results. The first one is an effective version of
Koll\'ar-Matsusaka theorem due to Demailly and Siu. We follow the
presentation in \cite{Laz2}.

\begin{theorem}[Koll\'ar-Matsusaka, Demailly, Siu]\label{matsu}
Let $X$ be a smooth projective variety. Let $L$ be an ample divisor
and $B$ be a nef divisor. Then for any $m\geq M$ the divisor $mL-B$
is very ample, where \bdism
M=(2n)^{(3^{n-1}-1)/2}\frac{\left(L^{n-1}\cdot(B+H)\right)^{(3^{n-1}+1)/2}\left(L^{n-1}\cdot
H\right)^{3^{n-2}(n/2-3/4)-1/4}}{\left(L^{n}\right)^{3^{n-2}(n/2-1/4)-1/4}},
\edism and $H=(n^{3}-n^{2}-n-1)(K_{X}+(n+2)L)$.
\end{theorem}

The second result is a theorem of M$^{\text{c}}$Kernan and it is an
important step in the study of boundedness of singular Fano varieties,
see \cite{Mck} and Corollary 1.8 in \cite{HMX}.

\begin{theorem}[M$^{\text{c}}$Kernan]\label{mck}
Fix $r$ and $n$ integers. Then there is a real number $M$ such that
if we have a pair $(X,D)$ where
\begin{enumerate}
\item X has dimension $n$,
\item $-(K_{X}+D)$ is big and nef,
\item $(X,D)$ is klt and
\item $r(K_{X}+D)$ is Cartier
\end{enumerate}
then $(-(K_{X}+D))^{n}\leq M$.
\end{theorem}

We can now prove Theorem \ref{logf} as stated in the Introduction
\ref{introduction}.

\begin{proof}[Proof of Theorem \ref{logf}]
Suppose that there exists $M_{0}$ as in the statement. If we prove
that there exists a positive integer $M$ such that
$-K_{X}+M(-K_{X}-D)$ is ample for any $(X,D)\in\bb$ then we are
done, because it implies that we can take
$\alpha_{0}=\frac{M}{M+1}$. In order to prove the existence of such
an integer we want to apply effective Koll\'ar-Matsusaka. Write
$L:=-(K_{X}+D)$ which is ample by hypothesis. Then if we define
$B=K_{X}+(n+1)L$ by Lemma \ref{lem1}, we know that $B$ is nef. For
the sake of clarity we do not keep track of all the constants which
appears in Theorem \ref{matsu}. In order to get an integer $M$ which
depends only on $L^{n}$ we need to bound all the intersection
numbers appearing in the formula with $L^{n}$. We can write $B=nL-D$
and $H=c(n)((n+1)L-D)$ where $c(n)=n^{3}-n^{2}-n-1$. Since $L$ is
ample and $D$ is effective we get the following inequalities \bdism
L^{n-1}\cdot B<nL^{n} \qquad \text{and} \qquad L^{n-1}\cdot
H<c(n)(n+1)L^{n}. \edism Plugging them in the formula of Theorem
\ref{matsu} we get that there are positive constants $a(n)$ and
$b(n)$ depending only on $n$ such that for any $m\geq
a(n)(L^{n})^{b(n)}$ the divisor $mL-B$ is ample. This is exactly
what we wanted at the beginning of the proof.

On the other hand suppose that we can find $\alpha_{0}$ as in $(1)$.
Choose a rational number $\alpha\in\left(\alpha_{0},1\right)$ and
let $r$ be its denominator. For any pair $(X,D)\in\bb$ we have that
$(X,\alpha D)$ is a klt pair and $r(K_{X}+\alpha D)$ is Cartier.
Then Theorem \ref{mck} applies and we get that there exists $M_{0}$
such that $(-(K_{X}+\alpha D))^{n}\leq M_{0}$ for any $(X,D)\in\bb$.
Since $-K_{X}-D+(1-\alpha)D=-K_{X}-\alpha D$ we get that also
$(-(K_{X}+D))^{n}$ is uniformly bounded.
\end{proof}

The following is a corollary of the first part of the proof of the
previous theorem.

\begin{corollary}
Fix a positive integer $n$. Let $X$ be a smooth projective variety
of dimension $n$ and let $D$ be an effective divisor such that
$-(K_{X}+D)$ is ample. Then there are positive integers $a(n)$ and
$b(n)$, depending only on $n$, such that \bdism
-\left(K_{X}+\frac{\alpha}{\alpha+1} D \right) \edism is ample for
any $\alpha\geq a(n)\left((-K_{X}-D)^{n}\right)^{b(n)}$. Furthermore
$a(n)$ and $b(n)$ are explicitly computable functions of $n$.
\end{corollary}

In order to state our next corollary we need a preliminary
definition.

\begin{definition}
Let $X$ be an irreducible projective variety and $L$ be a Cartier
divisor. Then the volume of $L$ is defined as follow \bdism
\vol(L):=\limsup_{m\to\infty}\frac{h^{0}(X,\oo_{X}(mL))}{m^{n}/n!}.
\edism
\end{definition}

By definition $\vol(L)>0$ if and only if $L$ is big. If $L$ is big
and nef then $\vol(L)=L^{n}$ but in general it is very hard to
compute the volume of a divisor.

\begin{corollary}
Let $\bb$ be a set of pairs $(X,D)$ such that $X$ is a smooth
projective variety of dimension $n$ and $D$ is a reduced effective
divisor with simple normal crossing support such that $-(K_{X}+D)$
is ample. Suppose there exists a positive integer $M$ such that
$\vol(-K_{X})\leq M$ for any $X\in\bb$. Then there exists a positive
number $\alpha_{0}<1$ such that $-(K_{X}+\alpha D)$ is ample for any
$\alpha\in(\alpha_{0},1]$ and any $(X,D)\in\bb$.
\end{corollary}

\begin{proof}
It simply follows from Theorem \ref{logf} and the fact that
$(-(K_{X}+D))^{n}\leq \vol(-K_{X})$.
\end{proof}

The above corollary tells us that we can choose $\alpha_{0}$ to be
independent of $D$.

Let us compute the volume of the anti-canonical divisors in Example
\ref{ex1}. We will use the following theorem of Zariski. See
\cite{Laz1} for details.

\begin{theorem}[Zariski decomposition]\label{zar}
Let $X$ be a smooth projective surface and let $L$ be a
pseudo-effective divisor. Then $L$ can be written uniquely as a sum
$L=P+N$ of $\qq$-divisors with the following properties:
\begin{enumerate}
\item $P$ is a nef divisor;
\item $N=\sum a_{i}E_{i}$ is effective, and if $N\neq 0$ then the intersection matrix $(E_{i}\cdot E_{j})_{i,j}$ is negative definite.
\item $P\cdot E_{i}=0$ for any component $E_{i}$ of $N$.
\end{enumerate}

Furthermore, the natural map  

\bdism
H^{0}(X,\oo_{X}(mL-\left\lceil
mN\right\rceil)) \rightarrow   H^{0}(X,\oo_{X}(mL)), 
\edism

is bijective for every $m\geq 1$.

\end{theorem}

$P$ and $N$ are called respectively the positive and negative parts
of $L$. Note that the last statement of Theorem \ref{zar} tells us
that $\vol(L)=P^{2}$.

\begin{example}
Let $X:=X_{e}$ and $D$ be as in the previous example. We can assume
$e\geq 3$ otherwise $-K_{X}$ is big and nef. Since $-K_{X}$ is big
we can consider the Zariski decomposition $-K_{X}=P+N$. In
particular $\vol(-K_{X})=P^{2}$. It is easy to see that \bdism
-K_{X}=\left(1+\frac{2}{e}\right)(C_{0}+ef)+\left(1-\frac{2}{e}\right)C_{0}
\edism is the Zariski decomposition of $-K_{X}$. Then \bdism
\vol(-K_{X})=P^{2}=\left(1+\frac{2}{e}\right)^{2}(C_{0}+ef)^{2}=(e+4)+\frac{4}{e}.
\edism
\end{example}

We now deal with the case $-K_{X}$ ample. The main question is the
following.

\begin{question}\label{qfano}
Let $X$ be a smooth variety with $-K_{X}$ ample and let $D$ be an
effective divisor. Is there a uniform $\alpha_{n}$, which depends
only on $n$, such that $-(K_{X}+\alpha D)$ is ample for any
$\alpha\in[0,\alpha_{n})$?
\end{question}

It is easy to see that in this case the answer is no even if we fix
$X$. Recall that in the previous case if we fix $X$ we can find a
uniform bound that works for any divisor on $X$ because it is enough
to bound $\vol(-K_{X})$.

\begin{example}\label{ex2}
Let $X=\pp^{n}$ and $D=\oo_{X}(d)$ for some $d>0$. Then
$-(K_{X}+\alpha D)$ is ample if and only if $\alpha<(n+1)/d$. Then
if $d$ is arbitrarily large we need to take $\alpha$ arbitrarily
small.
\end{example}

In order to settle Question \ref{qfano} we need a well known result
of Koll\'ar, Miyaoka and Mori. See Theorem 4.4 in \cite{KMM}.

\begin{theorem}[Koll\'ar-Miyaoka-Mori]\label{kmm}
Fix an integer $n$. Let $X$ be a smooth Fano variety of dimension
$n$. Then \bdism (-K_{X})^{n}\leq
(n+1)^{n}n^{(2^{n}-1)n(n+1)}\left(1+\frac{1}{n}+\frac{n+1}{n(n-1)}\right)^{n}.
\edism In particular the set of smooth Fano varieties of dimension
$n$ forms a bounded family.
\end{theorem}

Finally, we give the proof of the last theorem stated in the
Introduction \ref{introduction}.

\begin{proof}[Proof of Theorem \ref{logf2}]
Suppose that there exists $M_{0}$ as in the statement. We want to
apply Koll\'ar-Matsusaka's theorem to find an integer $M$ such that
$-K_{X}-D+M(-K_{X})$ is ample. Let $L:=-K_{X}$ and let
$B:=K_{X}+D+2nL$. We know by the cone theorem that $B$ is nef. We
need to control the following intersection numbers $B\cdot L^{n-1}$
and $H\cdot L^{n-1}$ where $H$ is as in the statement of Theorem
\ref{matsu}. We can write $B=(2n-1)L+D$ and $H=c(n)(n+1)L$ where
$c(n)=n^{3}-n^{2}-n-1$. Since $L$ is ample we have that $L^{n}\geq
1$ and Theorem \ref{kmm} tells us that there exists a function
$f(n)$ of $n$ such that $L^{n}\leq f(n)$. Then we can find the
desired integer $M$ applying Theorem \ref{matsu}.

Now assume that there exists $\alpha_{0}$ as in $(1)$. By assumption
we know that $H:=-(K_{X}+\alpha_{0} D)$ is a nef $\rr$-divisor. Then
\bdism \alpha_{0}D\cdot
(-K_{X})^{n-1}=(-K_{X}-H)\cdot(-K_{X})^{n-1}\leq (-K_{X})^{n},
\edism where the last inequality follows from the fact that $H$ is
nef and $-K_{X}$ is ample. Then it follows from Theorem \ref{kmm}
that we can uniformly bound $D\cdot(-K_{X})^{n-1}$.
\end{proof}

The first part of the proof gives the following.

\begin{corollary}
Let $(X,D)$ be a pair such that $X$ is a smooth Fano variety of
dimension $n$ and $D$ is a reduced effective divisor with simple
normal crossing support. Then there are two positive integers $a(n)$
and $b(n)$ depending only on $n$ such that
$-(K_{X}+\frac{1}{\alpha+1} D)$ is ample for any $\alpha\leq
a(n)\left(-D\cdot K_{X}^{n-1}\right)^{b(n)}$.\\
\end{corollary}

\end{document}